\documentclass{amsart}
\usepackage{}
\usepackage{mathrsfs}
\usepackage{graphicx} 

\usepackage{amssymb}
\usepackage{amsmath}
\usepackage{amsthm}
\usepackage[all,cmtip]{xy}

\newtheorem{thm}{Theorem}[section]
\newtheorem{cor}[thm]{Corollary}
\newtheorem{lem}[thm]{Lemma}
\newtheorem{prop}[thm]{Proposition}

\theoremstyle{definition}
\newtheorem{defn}[thm]{Definition}

\theoremstyle{remark}
\newtheorem{rem}[thm]{Remark}
\numberwithin{equation}{section}

\newcommand{\Z}{\mathbb Z}

\newcommand{\fix}{\mathrm{Fix}\,}

\newcommand{\rk}{\mathrm{rk}}

\newcommand{\aut}{\mathrm{Aut}}
\newcommand{\edo}{\mathrm{End}}

\newcommand{\B}{\mathcal{B}}

\begin{document}

\title{Fixed subgroups in direct products of surface groups of Euclidean type}

\author{Jianchun Wu}
\address{Department of Mathematics, Soochow University, Suzhou 215006, CHINA}
\email{wujianchun@suda.edu.cn}

\author{Enric Ventura}
\address{Departament de Matem\`atiques, Universitat Polit\`ecnica de Catalunya, CATALONIA} \email{enric.ventura@upc.edu}

\author{Qiang Zhang}
\address{School of Mathematics and Statistics, Xi'an Jiaotong University, Xi'an 710049,
CHINA} \email{zhangq.math@mail.xjtu.edu.cn}



\begin{abstract}
We give an explicit characterization of which direct products $G$ of surface groups of Euclidean type satisfy that the fixed subgroup of any automorphism (or endomorphism) of $G$ is compressed, and of which is it always inert.
\end{abstract}
\maketitle

\section{Introduction}

For a finitely generated group $G$, let $\rk(G)$ denote its rank, i.e., the minimal number of the generators of $G$. Let endomorphisms of $G$ act on the left, i.e., $g\mapsto \phi g$, and denote the monoid of endomorphisms of $G$ by $\edo(G)$, and the group of automorphisms by $\aut(G)$. For an arbitrary family $\emptyset \neq \B\subseteq \edo(G)$, the \emph{fixed subgroup} of $\B$ is
 $$
\fix \B :=\{g\in G \mid \phi g=g,\,\, \forall \phi\in \B\} =\bigcap_{\phi\in
\mathcal{B}} \fix\phi \leqslant G.
 $$

\begin{defn}
A subgroup $H\leqslant G$ is said to be:
\begin{itemize}
\item[(i)] \emph{inert in $G$} if $\rk(H\cap K)\leqslant \rk(K)$ for every (finitely generated) $K\leqslant G$;
\item[(ii)] \emph{compressed in $G$} if $\rk(H)\leqslant \rk(K)$ for every (finitely generated) $K$ with $H\leqslant K\leqslant G$;
\item[(iii)] \emph{c-bounded in $G$} if $\rk(H) \leqslant c \cdot \rk(G)$; for simplicity, when $H$ is 1-bounded in $G$, we just say it is \emph{bounded in $G$}.
\end{itemize}
\end{defn}

In~\cite{BH}, Bestvina--Handel solved the famous Scott's conjecture that the fixed subgroup of an automorphism of a free group $F$ is bounded in $F$:

\begin{thm}[Bestvina--Handel, \cite{BH}]\label{BH}
For any automorphism $\phi$ of a free group $F$, $\rk(\fix\phi) \leqslant \rk(F)$.
\end{thm}

The notions of inertia and compression were introduced in Dicks--Ventura~\cite{DV}, where the authors proved the following extension of Bestvina--Handel's result.

\begin{thm}[Dicks--Ventura, \cite{DV}]\label{inj fixed subgp of free gp inert}
Let $F$ be a finitely generated free group, and let $\B$ be a family of injective endomorphisms of $F$. Then, $\fix\B$ is inert in $F$.
\end{thm}

It is obvious that inertia implies compression, and compression implies boundedness; also, a finite intersection of inert subgroups in $G$ is again inert in $G$. There has been several papers in the literature studying these properties for fixed subgroups of automorphisms and endomorphisms of free groups, surface groups, and finite direct products of them. We list some as follows.

\begin{thm}[Bergman, \cite{B}]
Let $F$ be a finitely generated free group, and let $\B \subseteq \edo(F)$. Then, $\fix\B$ is bounded in $F$.
\end{thm}

\begin{thm}[Martino--Ventura, \cite{MV}]\label{compression of fixed subgp in free gp}
Let $F$ be a finitely generated free group, and let $\B\subseteq \edo(F)$. Then, $\fix \B$ is compressed in $F$.
\end{thm}

It is still open whether $\fix\phi$ is inert in $F$ for every $\phi\in \edo(F)$, a question asked first in~\cite[Problem~5]{DV}, and known as the \emph{inertia conjecture} (a small step into this direction can be found in~\cite{ZVW}).

\bigskip

For surface groups a lot is know too. A \emph{surface group} is the fundamental group $\pi_1(X)$ of a connected, possibly puctured, surface $X$. To fix notation, we denote $\Sigma_g$ the closed orientable surface of genus $g\geqslant 1$, and
 $$
S_g= \pi_1(\Sigma_g )=\langle a_1,b_1,\ldots,a_g,b_g \mid [a_1, b_1]\cdots [a_g, b_g]
\rangle
 $$
its fundamental group; for the non-orientable case, we denote $N\Sigma_k$ the connected sum of $k\geqslant 1$ projective planes, and
 $$
NS_k =\pi_1(N\Sigma_k)=\langle a_1,a_2,\ldots,a_k \mid a_1^2 \cdots a_k^2\rangle
 $$
its fundamental group. With one or more punctures, we get the finitely generated free groups $F_r$. Straightforward calculations show that $NS_2 \simeq \langle a,b \mid bab^{-1}a\rangle$; we shall work with this presentation. 

Nielsen, Jaco--Shalen, and Zieschang gave the following results.

\begin{thm}[Nielsen, \cite{N1, N2}; Jaco--Shalen, \cite{JS}]
Let $G$ be a surface group. For any $\phi\in \aut(G)$, $\fix\phi$ is bounded in $G$.
\end{thm}

\begin{thm}[Zieschang, \cite{Z}]
Let $G$ be a closed surface group with $\chi(G)<0$. For any non-epimorphic endomorphism $\phi\in \edo(G)$, $\fix\phi$ is $\frac{1}{2}$-bounded in $G$.
\end{thm}

Some alternative proofs can also be found in~\cite{JWZ}. Some years ago, these two theorems were extended in~\cite{WZ} to families of endomorphisms.

\begin{thm}[Wu--Zhang, \cite{WZ}]\label{surface inert}
Let $G$ be a closed surface group with $\chi(G)<0$. For any $\B\subseteq \aut(G)$, $\fix \B$ is inert in $G$.
\end{thm}

And, recently, the analog of Theorem~\ref{compression of fixed subgp in free gp} has been obtained in~\cite{ZVW} for surface groups.

\begin{thm}[Zhang--Ventura--Wu, \cite{ZVW}]
Let $G$ be a surface group and $\B\subseteq \edo(G)$. Then, $\fix \B$ is compressed in $G$.
\end{thm}

To our knowledge, there are only a few results of this type for fundamental groups of 3-manifolds; see~\cite{LW, Zhang1, Zhang2}.

\bigskip

Recently, Zhang--Ventura--Wu started investigating in~\cite{ZVW} these same questions but within the family of finite direct products of free and surface groups, i.e., groups of the form $G=G_1\times \cdots \times G_n$, where each building block $G_i$ is either 1, $F_r$, $S_g$, or $NS_k$. Here, the situation is more complicated and interesting new phenomena show up. In order to study the behaviour of fixed subgroups, it becomes relevant to distinguish between the \emph{Euclidean} building blocks, 1, $\Z$, $\Z^2$, $\Z_2$, and $NS_2$ (corresponding, respectively, to the sphere, two-punctured sphere, torus, projective plane, and Klein bottle), and the \emph{hyperbolic} building blocks, $F_r$ with $r\geqslant 2$, $S_g$ with $g\geqslant 3$, and $NS_k$ with $k\geqslant 3$ (corresponding, respectively, to a $r$-punctured disc, a connected sum of $g$ tori, and a connected sum of $k$ projective planes). Geometrically, this is, precisely, distinguishing whether the surface has non-negative or negative Euler characteristic, respectively; and algebraically, it is distinguishing whether the building block has non-trivial or trivial center, respectively (see~\cite[Lemma~4.2]{ZVW}).

Let $G$ be a finite direct product of surface groups. \cite[Theorem~4.8]{ZVW} states that mixing the two types of building blocks is the unique possible way to break boundedness of fixed subgroups of automorphisms of $G$; more precisely,

\begin{thm}[Zhang--Ventura--Wu, \cite{ZVW}]
Let $G=G_1\times \cdots \times G_n$, where each $G_i$ is a surface group. Then, $\fix \phi$ is bounded in $G$ for all $\phi \in \aut(G)$ if and only if the $G_i$'s are all Euclidean, or all hyperbolic.
\end{thm}

Then, the authors gave necessary conditions on $G$ for all fixed subgroups of automorphisms of $G$ to be compressed in $G$, or to be inert in $G$:

\begin{thm}[Zhang--Ventura--Wu, \cite{ZVW}]\label{products compressed}
Let $G=G_1\times \cdots \times G_n$, where each $G_i$ is a surface group. If $\fix \phi$ is compressed in $G$ for every $\phi\in \aut(G)$, then $G$ must be of one of the following forms:
\begin{itemize}
\item[(euc1)] $G=\Z^p \times \Z_2^q$ for some $p,q\geqslant 0$; or
\item[(euc2)] $G=NS_2\times \Z_2^q$ for some $q\geqslant 0$; or
\item[(euc3)] $G=NS_2\times \Z^p \times \Z_2$ for some $p\geqslant 1$; or
\item[(euc4)] $G=NS_2^{\, \ell}\times \Z^p$ for some $\ell\geqslant 1$, $p\geqslant 0$; or
\item[(hyp1)] $G=F_r\times NS_3^{\ell}$ for some $r\geqslant 2$, $\ell\geqslant 0$; or
\item[(hyp2)] $G=S_g\times NS_3^{\ell}$ for some $g\geqslant 2$, $\ell\geqslant 0$; or
\item[(hyp3)] $G=NS_k\times NS_3^{\ell}$ for some $k\geqslant 3$, $\ell\geqslant 0$.
\end{itemize}
\end{thm}

\begin{thm}[Zhang--Ventura--Wu, \cite{ZVW}]\label{products inert}
Let $G=G_1\times \cdots \times G_n$, where each $G_i$ is a surface group. If $\fix \phi$ is inert in $G$ for every $\phi\in \aut(G)$, then $G$ must be of the form \emph{(euc1)}, or \emph{(euc2)}, or \emph{(euc3)}, or \emph{(euc4)}, or
\begin{itemize}
\item[(hyp1')] $G=F_r$ for some $r\geqslant 2$; or
\item[(hyp2')] $G=S_g$ for some $g\geqslant 2$; or
\item[(hyp3')] $G=NS_k$ for some $k\geqslant 3$.
\end{itemize}
\end{thm}

\bigskip

In the present paper we complete a full characterization of compression and inertia for the Euclidean case (which happens to be the same for automorphisms and for endomorphisms), by proving the following improvements of Theorems~\ref{products compressed} and~\ref{products inert}:

\begin{thm}\label{main-thm}
Let $G=G_1\times \cdots \times G_n$, where each $G_i$ is a surface group of Euclidean type. Then, the following are equivalent:
\begin{itemize}
\item[(a)] $\fix \phi$ is compressed in $G$ for every $\phi\in \edo(G)$;
\item[(b)] $\fix \phi$ is compressed in $G$ for every $\phi\in \aut(G)$;
\item[(c)] $G$ is of the form \emph{(euc1)}, or \emph{(euc2)}, or \emph{(euc4)}.
\end{itemize}
\end{thm}

\begin{thm}\label{main-thm2}
Let $G=G_1\times \cdots \times G_n$, where each $G_i$ is a surface group of Euclidean type. Then, the following are equivalent:
\begin{itemize}
\item[(a)] $\fix \phi$ is inert in $G$ for every $\phi\in \edo(G)$;
\item[(b)] $\fix \phi$ is inert in $G$ for every $\phi\in \aut(G)$;
\item[(c)] $G$ is of the form \emph{(euc1)}, or \emph{(euc2)}.
\end{itemize}
\end{thm}

As an immediate corollary, together with Theorems~\ref{products inert}, \ref{inj fixed subgp of free gp inert}, and~\ref{surface inert}, we also get a full characterization of inertia for both types of surfaces, but only concerning automorphisms (the endomorphism case depends on the inertia conjecture, which is still open):

\begin{thm}\label{main-thm3}
Let $G=G_1\times \cdots \times G_n$, where each $G_i$ is a surface group. Then, the following are equivalent:
\begin{itemize}
\item[(a)] $\fix \phi$ is inert in $G$ for every $\phi\in \aut(G)$;
\item[(b)] $G$ is of the form \emph{(euc1)}, or \emph{(euc2)}, or \emph{(hyp1')}, or \emph{(hyp2')}, or \emph{(hyp3')}.
\end{itemize}
\end{thm}

The complete characterization of compression for mixed type of surface groups remains also an open problem, which seems somehow trickier.

\begin{proof}[Proof of Theorem~\ref{main-thm}]
(a)$\Rightarrow$ (b) is trivial.

For (b) $\Rightarrow$ (c), we apply Theorem~\ref{products compressed} and remove the case (euc3) by using the counter-example given in Proposition~\ref{euc3} from Section~3.

Finally, to see that (c) $\Rightarrow$ (a), we first note that, by elementary results about abelian groups, every subgroup of $G$ is inert in $G$ (and hence compressed in $G$) in the case (euc1); the same is true for the case (euc2), as proved in Proposition~\ref{euc2} from Section~2; finally, for the case (euc4), in Section~4 we use a property of fixed points of endomorphisms to show that $\fix \phi$ is compressed in $G$ for any $\phi\in \edo(G)$; see Corollary~\ref{euc4}.
\end{proof}

\begin{proof}[Proof of Theorem~\ref{main-thm2}]
(a)$\Rightarrow$ (b) is trivial.

For (b) $\Rightarrow$ (c), we apply Theorem~\ref{main-thm} and remove the case (euc4) by using the counter-example given in Proposition~\ref{no-comp} from Section~4.

Finally, in abelian groups every subgroup is inert, and the same is true for the case (euc2), as proved in Proposition~\ref{euc2}; this proves (c) $\Rightarrow$ (a).
\end{proof}

\section{The case (euc2)}

The goal of this section is to show that every subgroup of
 $$
G=NS_2\times \Z_2^q =\langle a, b \mid b a b^{-1} a\rangle \times \prod_{j=1}^q \langle d_j \mid d_j^2\rangle,
 $$
where $q\geqslant 0$, is inert (and hence compressed) in $G$. We need the following lemma, which follows from standard arguments.

\begin{lem}\label{dp}
Let $A,B$ be two groups, let $G=A\times B$, let $\pi \colon G\twoheadrightarrow A$ be the natural projection (with $\ker\pi =B$), and let $H\leqslant G$ be a subgroup such that $H\cap B\leqslant Z(H)$. If $\pi_{|H}\colon H\twoheadrightarrow \pi H\leqslant A$ splits then $H\simeq \pi H\times (H\cap B)$.
\end{lem}

\begin{proof}
Let $H\leftarrow \pi H \colon \alpha$ be a morphism such that $\pi \alpha =Id_{\pi H}$, and consider
 $$
\begin{array}{rcl} \Psi \colon H & \to & \pi H\times (H\cap B) \\ h & \mapsto & \Big( \pi h,\, (\alpha\pi h)^{-1}\cdot h\Big). \end{array}
 $$
It is well-defined since $\pi ((\alpha\pi h)^{-1}\cdot h)=(\pi\alpha\pi h)^{-1}\cdot (\pi h)=1$ and so, $(\alpha\pi h)^{-1}\cdot h\in \ker \pi_{|H}=H\cap B$. It is a morphism since
 $$
\begin{array}{rcl}
\Psi(h_1h_2) &=& \big( \pi (h_1h_2),\, (\alpha\pi (h_1h_2))^{-1}\cdot (h_1h_2)\big) \\ &=& \big( (\pi h_1)\cdot (\pi h_2),\, (\alpha\pi h_2)^{-1}\cdot (\alpha\pi h_1)^{-1}\cdot h_1\cdot h_2\big) \\ &=& \big( (\pi h_1)\cdot (\pi h_2),\, (\alpha\pi h_1)^{-1}\cdot h_1\cdot (\alpha\pi h_2)^{-1}\cdot h_2\big) \\ &=& \Psi h_1 \cdot \Psi h_2,
\end{array}
 $$
where the middle step is valid because $(\alpha\pi h_1)^{-1}\cdot h_1\in H\cap B\leqslant Z(H)$, by hypothesis. Finally, it is straightforward to see that $\pi H\times (H\cap B) \to H$, $(x,y)\mapsto (\alpha x)\cdot y$ is a homomorphism and is the inverse of $\Psi$. Hence, $\Psi$ is bijective.
\end{proof}

\begin{prop}\label{euc2}
Every subgroup $H\leqslant G=NS_2\times \mathbb{Z}_2^q$, $q\geqslant 0$, is inert (and hence compressed) in $G$.
\end{prop}

\begin{proof}
First, note that any subgroup of $NS_2$ is isomorphic to either $1,\Z,\Z^2$, or $NS_2$; it is then immediate that any subgroup of $NS_2$ is inert in $NS_2$; this is the case $q=0$.

Suppose $q\geqslant 1$, fix an arbitrary $H\leqslant G=NS_2\times \Z_2^q$, and let us see it satisfies the conditions of Lemma~\ref{dp}. The fact $H\cap \Z_2^q \leqslant Z(H)$ is obvious because $\Z_2^q$ is abelian. To see that the restricted short exact sequence $1\to H\cap \Z_2^q \to H\to \pi H \to 1$ splits, let us distinguish the isomorphism type of $\pi H\leqslant NS_2$: if $\pi H\simeq 1,\, \Z$ then it is free and the sequence clearly splits; if $\pi H\simeq \Z^2,\, NS_2$, take two elements $u,v\in \pi H$ which present $\pi H$, say $\pi H\simeq \langle u,v \mid uv=vu \rangle$ or $\pi H\simeq \langle u,v \mid vuv^{-1}u \rangle$, respectively, and choose arbitrary preimages $ue_u,\, ve_v\in H$, where $e_u,e_v\in \Z_2^q$ and so satisfy $e_u^2=e_v^2=1$. We claim that, in both cases, $\alpha \colon \pi H\to H$, $u\mapsto ue_u$, $v\mapsto ve_v$ splits the sequence. In fact, it is a well-defined morphism because, if $uv=vu$ in $\pi H$ then certainly $ue_u \cdot ve_v=ve_v \cdot ue_u$ in $G$; and if $vuv^{-1}u=1$ in $\pi H$ then certainly $ve_v\cdot ue_u \cdot (ve_v)^{-1}\cdot ue_u = vuv^{-1}u \cdot e_u^2=1$ in $G$. Hence, by Lemma~\ref{dp}, $H\simeq \pi H\times (H\cap \Z_2^q)$.

Now take an arbitrary $R\leqslant G$. By the previous paragraph, we also have $R\simeq \pi R\times (R\cap \Z_2^q)$ and $H\cap R\simeq \pi (H\cap R)\times (H\cap R\cap \Z_2^q)$. Then,
 $$
\rk(H\cap R)=\rk(\pi (H\cap R))+\rk(H\cap R\cap \Z_2^q)\leqslant \rk(\pi R)+\rk(R\cap \Z_2^q)=\rk (R),
 $$
where the inequality $\rk(\pi (H\cap R))\leqslant \rk(\pi R)$ is true because $\pi (H\cap R)\leqslant \pi R\leqslant NS_2$, and subgroups of $NS_2$ are compressed in $NS_2$.
\end{proof}

\begin{rem}\label{remark}
We want to direct reader's attention to the fact that elements in $\Z_2^q$ all have order 2; this played a crucial role in the above argument because, whatever preimage $ue_u\in H$ we choose for $u\in \pi H$ satisfies $e_u^2=1$, and this was necessary to make $\alpha$ well-defined in the case $\pi H\simeq NS_2$. This small detail is important in light of the following example.

We observe that the result stated in Proposition~\ref{euc2} is \emph{not} true if we remove this order two relation: in general, subgroups of $G=NS_2 \times \Z =\langle a, b \mid b a b^{-1} a\rangle \times \langle c\rangle$ are \emph{no longer} inert, neither even compressed in $G$. Consider $R=\langle ac,b\rangle \leqslant G$; it contains $a^{-1}c=b(ac)b^{-1}$ and so, $a^2=(ac)(a^{-1}c)^{-1}$ and $c^2=(ac)(a^{-1}c)$; therefore, $H=\langle a^2, b^2, c^2\rangle =\langle a^2, b^2\rangle \times \langle c^2\rangle\simeq \Z^3$ is contained in $R\leqslant G$ and so, it is not compressed in $G$. (Note that, inserting this $R$ into the argument in Proposition~\ref{euc2}, the restriction of the projection $\pi \colon G\twoheadrightarrow NS_2$ to $R$, say $\pi_{|R}\colon R\twoheadrightarrow \pi R=NS_2$, $ac\mapsto a$, $b\mapsto b$ \emph{does not} split: all elements $r\in R$ satisfy $|r|_c=|r|_a$ (which only makes sense modulo 2) and so all preimages of $a$ in $R$ have odd, and so non-trivial, $c$-exponent; hence, the relation $bab^{-1}a$ can never be preserved up in $R$.)
\end{rem}

This remark tells us that a general approach like we did in the case (euc2) will not work in the case (euc4). So, in Section~4, we will be forced to use specific properties of fixed subgroups of endomorphisms.

\section{The case (\lowercase{euc3})}

In this section, we consider the case (euc3), $G=NS_2\times \Z^p \times \Z_2$ for some $p\geqslant 1$, and exhibit an automorphism $\phi\in \aut(G)$ such that $\fix\phi$ is not compressed (and hence not inert) in $G$.

Put $G=NS_2\times \Z^p \times \Z_2=\langle a, b \mid b a b^{-1} a\rangle\times \prod_{i=1}^p \langle c_i\rangle\times \langle d \mid d^2\rangle$. Since $ba^n =a^{-n}b$ for $n\in \Z$, and $c_i, d\in Z(G)$ for $i=1,\ldots ,p$, we see that each $g\in G$ can be written as $g=a^{k_0}b^{k_1} c_1^{n_1} \cdots c_p^{n_p} d^{k_2}$, for some unique $k_0, k_1, n_1, \ldots ,n_p\in \Z$ and unique $k_2 \in \Z_2$.

Consider $\phi\colon G \to G$ given by $a\mapsto ad$, $b\mapsto ba$, $c_1 \mapsto c_1d$, $c_i \mapsto c_i^{-1}$  $(i=2,\ldots ,p)$, $d\mapsto d$.

\begin{prop}\label{euc3}
$\phi$ is a well-defined automorphism of $G=NS_2\times \Z^p\times \Z_2$, $p\geqslant 1$, and $\fix\phi$ is not compressed in $G$.
\end{prop}

\begin{proof}
It is straightforward to see that the defining relations from $G$ are preserved and so, $\phi$ is a well-defined endomorphism of $G$. To see it is an automorphism, one just checks that $\psi\colon G \to G$, $a\mapsto ad$, $b\mapsto ba^{-1}d$, $c_1 \mapsto c_1d$, $c_i \mapsto c_i^{-1}$  $(i=2,\ldots ,p)$, $d\mapsto d$ is again well defined and satisfies $\phi\psi=\psi\phi=Id$.

Now let us prove that $\fix\phi=\langle a^2, b^2, ac_1, d\rangle$. The inclusion $\geqslant$ is straightforward. Let $g=a^{k_0}b^{k_1} c_1^{n_1} c_2^{n_2}\cdots c_p^{n_p} d^{k_2}$ be fixed by $\phi$; we have
 $$
a^{k_0}b^{k_1} c_1^{n_1} c_2^{n_2}\cdots c_p^{n_p} d^{k_2} =g=\phi g= a^{k_0-\delta(k_1)}b^{k_1} c_1^{n_1} c_2^{-n_2}\cdots c_p^{-n_p} d^{k_0+n_1+k_2},
 $$
where $\delta(k)$ is 1 if $k$ is odd, and 0 if $k$ is even. So, $n_2=\cdots=n_p=0$, and both $k_1$ and $k_0+n_1$ must be even, therefore $k_0-n_1$ is also even and $g=a^{k_0}b^{k_1} c_1^{n_1} c_2^{n_2}\cdots c_p^{n_p} d^{k_2} =a^{k_0-n_1}b^{k_1} (ac_1)^{n_1} d^{k_2}\in \langle a^2, b^2, ac_1, d\rangle$ (recall that $b^2$ commutes with $a$).

Finally, note that $\fix\phi =\langle a^2, b^2, ac_1\rangle \times \langle d\rangle\cong \Z^3\times \Z_2$ has rank 4 but, $\fix\phi \leqslant \langle ac_1, b, d\rangle$ since $a^2=(ac_1) b (ac_1)^{-1} b^{-1}$. Hence, $\fix\phi$ is not compressed in $G$.
\end{proof}

\section{The case (euc4)}

In this section, we concentrate on case (euc4), i.e.,
 $$
G=NS_2^{\, \ell}\times \Z^p =\prod_{i=1}^\ell \langle a_i, b_i| b_i a_i b_i^{-1} a_i\rangle\times \prod_{j=1}^p \langle c_j\rangle,
 $$
$\ell\geqslant 1$, $p\geqslant 0$. We shall prove that the fixed subgroup of any endomorphism of $G$ is compressed in $G$ (despite the fact that, except for the case $\ell=1$, $p=0$, $G$ contains some non compressed subgroups; see Remark~\ref{remark}). When $\ell+p\geqslant 2$, we shall also exhibit an automorphism $\phi \in \aut(G)$ such that $\fix\phi$ is not inert in $G$.

With this notation, consider $G\unrhd N=\langle a_1,\ldots ,a_\ell\rangle\cong \Z^\ell$, and the commutator subgroup, $G\unrhd G'=\langle a_1^2,\ldots ,a_\ell^2 \rangle\cong \Z^\ell$. Observe that every $g\in G'$ has a unique square root in $G$, which happens to belong to $N$, and we denote it by $\sqrt{g}\in N$: in fact, in each $NS_2$ direct factor of $G$, the square of an element with non-zero total number of $b_i$'s has non-zero total number of $b_i$'s. (In general, however, roots are not unique in $NS_2$, since $(a_i^r b_i)^2=b_i^2$ for every $r\in \Z$.)

\begin{lem}\label{compr}
Let $\varphi \colon G_1 \to G_2$ be a group morphism. If a subgroup $H\leqslant G_1$ satisfies $\rk(\varphi (H))=\rk(H)$, and $\varphi(H)$ is compressed in $Im(\varphi)$ then, $H$ is compressed in $G_1$.
\end{lem}

\begin{proof}
For every $H\leqslant K\leqslant G_1$, we have $\varphi (H)\leqslant \varphi(K)\leqslant Im(\varphi)\leqslant G_2$; therefore, $\rk(H)=\rk(\varphi(H))\leqslant \rk(\varphi(K))\leqslant \rk(K)$.
\end{proof}

\begin{prop}\label{compressed}
Let $G=NS_2^\ell\times \mathbb{Z}^p$, $\ell\geqslant 1$, $p\geqslant 0$. If a subgroup $H\leqslant G$ satisfies that $\sqrt{h}\in H$ for every $h\in H\cap G'$, then $H$ is compressed in $G$.
\end{prop}

\begin{proof}
Let $\{h_1,\ldots ,h_k\}$ be a free-abelian basis of $H\cap G'\simeq \Z^k$, $k\leqslant \ell$. From the hypothesis on $H$, it follows that $\{\sqrt{h_1}, \ldots ,\sqrt{h_k}\}$ is a free-abelian basis of $H\cap N\simeq \Z^k$. Therefore, $(H\cap N)/(H\cap G') \simeq \Z_2^k$.

On the other hand $H/(H\cap N)\leqslant G/N\simeq \Z^{\ell+p}$, so $H/(H\cap N)\simeq \Z^r$ for some $r\leqslant \ell+p$; also, $H/(H\cap G')\leqslant G/G'\simeq \Z_2^{\ell}\times \Z^{\ell+p}$, so $H/(H\cap G')\simeq \Z_2^{s}\times \Z^{t}$ for some $s\leqslant \ell$ and $t\leqslant \ell+p$. But
 $$
\Z^r\simeq H/(H\cap N)\simeq \frac{H/(H\cap G')}{(H\cap N)/(H\cap G')} \simeq \frac{\Z_2^s\times \Z^t}{\Z_2^k}
 $$
so, $s=k$ and $t=r$. Then, for the abelianization map $\rho \colon G\to G/G'$, we have $\rk (\rho(H))=\rk(H/(H\cap G'))=s+t=k+r=\rk(H\cap N)+\rk(H/(H\cap N))\geqslant \rk(H)$ and hence, $\rk(\rho(H))=\rk(H)$. By Lemma~\ref{compr}, $H$ is compressed in $G$.
\end{proof}

\begin{cor}\label{euc4}
Let $G=NS_2^\ell\times \mathbb{Z}^p$, $\ell\geqslant 1$, $p\geqslant 0$. For every $\phi\in \edo(G)$, $\fix\phi$ is compressed in $G$.
\end{cor}

\begin{proof}
By Proposition~\ref{compressed}, it suffices to show that $\sqrt{h}\in \fix\phi$, for every $h\in \fix\phi \cap G'$. And this follows from the above observation that $\sqrt{h}$ is the unique element in $G$ with square equal to $h$ (hence, if $h$ is fixed by $\phi$, $\sqrt{h}$ must also be fixed by $\phi$).
\end{proof}

Finally, in this case (euc4), we give examples of automorphisms $\phi \in \aut(G)$ such that $\fix\phi$ is not inert in $G$ (except, of course, for the case $\ell=1$, $p=0$, which is included in (euc2)).

\begin{prop}\label{no-comp}
Let $G=NS_2^{\ell}\times \Z^p$, $\ell\geqslant 1$, $\ell +p\geqslant 2$. There exists an automorphism $\phi \in \aut(G)$ such that  $\fix\phi$ is not inert in $G$.
\end{prop}

\begin{proof}
The (necessary) condition $\ell+p\geqslant 2$ (equivalently, $p\geqslant 1$ or $\ell\geqslant 2$) just means we have at least two direct factors in $G$ (necessary in order to avoid the case $G=NS_2$, for which we already know that every subgroup is inert).

Suppose $\ell=p=1$, i.e., $G=NS_2\times \Z=\langle a,b \mid bab^{-1}a\rangle \times \langle c\rangle$. Consider $\phi_1\colon G\to G$, $a\mapsto a$, $b\mapsto ba$, $c\mapsto c$; straightforward calculations show that it is a well-defined automorphism $\phi_1\in \aut(G)$, and that $\fix\phi_1=\langle a, b^2, c\rangle\cong \Z^3$. But this is not inert in $G$ because $\fix\phi_1 \cap\langle ac, b\rangle= \langle ac, a^2, b^2\rangle\cong \Z^3$ (the inclusion $\geqslant$ follows from direct inspection, and $\leqslant$ is a consequence of $\langle ac, a^2, b^2\rangle$ being an index 2 subgroup of $\langle a, b^2, c\rangle$).

Now, suppose $\ell=2$, $p=0$, i.e., $G=(NS_2)^2=\langle a_1,b_1 \mid b_1a_1b_1^{-1}a_1\rangle \times \langle a_2,b_2 \mid b_2a_2b_2^{-1}a_2\rangle$. Consider $\phi_2\colon G\to G$, $a_1\mapsto a_1$, $b_1\mapsto b_1a_1$, $a_2\mapsto a_2$, $b_2\mapsto b_2^{-1}$; again, straightforward calculations show that it is a well-defined automorphism $\phi_2\in \aut(G)$, and that $\fix\phi_2=\langle a_1, b_1^2, a_2\rangle\cong \Z^3$. By the same argument as in the previous paragraph, this is not inert in $G$.

The general case follows easily by taking the identity in all the subsequent direct factors behind the first two.
\end{proof}

\begin{rem}
The concepts of compression and inertia are quite delicate and sensible to the ambient group. In the last sentence of the previous proof, we used the straightforward fact that, for $\phi\in \aut(G_1)$, if $\fix(\phi)$ is not inert in $G_1$ then $\fix(\phi\times id)$ is not inert in $G_1\times G_2$. However, the converse is not true: we proved that $\langle a, b^2\rangle\times \langle c\rangle$ is \emph{not} inert in $NS_2\times \Z$, while $\langle a, b^2\rangle$ \emph{is} inert in $NS_2$, by Proposition~\ref{euc2}.
\end{rem}

\noindent\textbf{Acknowledgements.} The first named author is partially supported by the National Natural Science Foundation of China (No. 11571246); the second by the Spanish Agencia Estatal de Investigaci\'on, through grant MTM2017-82740-P (AEI/ FEDER, UE), and  also by the ``Mar\'{\i}a de Maeztu'' Programme for Units of Excellence in R\&D (MDM-2014-0445); finally, the third named author is partially supported by the National Natural Science Foundation of China (No. 11771345).

\end{document}